 \long\def\drop#1{}
\documentclass[reqno]{article}
\pdfoutput=1
\usepackage{a4wide}
\usepackage{amsmath,amsfonts,amsthm,amssymb}
\usepackage{mathrsfs,eucal}
\usepackage[normalem]{ulem}
\usepackage{geompsfi}
\usepackage{color}
\usepackage{enumerate}
\newcommand{\nnabla}{\nabla\kern-2pt}

\newtheorem{theorem}{Theorem}
\newtheorem{corollary}[theorem]{Corollary}
\newtheorem{lemma}[theorem]{Lemma}

\newtheoremstyle{remark}
{3pt}
{3pt}
{}
{}
{\bfseries}
{.}
{.5em}
{}

\theoremstyle{remark}
\newtheorem*{remark}{Remark}
\def\rme{{\mathrm e}}

\newcommand{\lambdaO}{{\lambda_\Omega}}

\newcommand{\restr}[1]{\lower3pt\hbox{$|_{#1}$}}

\newcommand{\GG}[1]{{\color{blue}{#1}}}

\newcommand{\Cb}{{\mathrm{C^0}}}
\renewcommand{\d}{{\mathrm d}}

\newcommand{\nablae}{{\nabla_{\kern-2pt\e\kern1pt}}}
\newcommand{\rhoe}{{\rho_\e}}
\newcommand{\ue}{{u_\e}}
\newcommand{\dive}{{\div\nolimits_{\kern-1pt\e}}}
\let\e\varepsilon
\newcommand{\xx}{{\mbox{\boldmath$x$}}}

\newcommand{\yy}{{\mbox{\boldmath$y$}}}

\newcommand{\Dom}{D}
\newcommand{\cDom}{{\overline D}}
\newcommand{\sDom}{D}
\def\R{{\mathbb R}}

\def\D{{\mathscr D}}
\def\T{\Omega}
\def\P{{\mathscr P}}
\def\M{{\mathscr M}}
\renewcommand{\a}{\mathfrak a}
\renewcommand{\b}{\mathfrak b}
\newcommand{\q}{\mathfrak q}

\let\weakto\rightharpoonup
\def\pref#1{(\ref{#1})}

\newcommand{\ds}{\displaystyle}
\def\div{\mathop{\mathrm{div}}}
\newcommand{\Ld}{\mathcal{L}^{d}}
\renewcommand{\L}{\mathcal{L}}

\newcommand{\Domain}{\mathsf{Dom}}

\begin{document}

\title{From diffusion to reaction via $\Gamma$-convergence \vspace{0.1cm}}
\author{Mark A. Peletier\footnote{Department of Mathematics and Institute for Complex Molecular Systems, Technische Universiteit Eindhoven, The Netherlands, m.a.peletier@tue.nl} \and Giuseppe Savar\'e\footnote{Dipartimento di Matematica F.Casorati, Universit\`a degli studi di Pavia, Italy, giuseppe.savare@unipv.it} \and Marco Veneroni\footnote{Fakult\"at f\"ur Mathematik, Technische Universit\"at Dortmund, Germany, marco.veneroni@math.uni-dortmund.de}}
\date{\today}
\maketitle

\abstract{We study the limit of {\em high activation energy} of a special Fokker-Planck equation, known as Kramers-Smoluchowski (K-S) equation. This equation governs the time evolution of the probability density of a particle performing a Brownian motion under the influence of a chemical potential $H/\e$. We choose $H$ having two wells corresponding to two chemical states $A$ and $B$. We prove that after a suitable rescaling the solution to (K-S) converges, in the limit of high activation energy ($\e\to0$), to the solution of a simple system modeling the diffusion of $A$ and $B$, and the reaction $A\rightleftharpoons B$. 

The aim of this paper is to give a rigorous proof of Kramer's formal derivation and 
to embed chemical reactions and diffusion processes in a common variational framework
which allows to derive the former as a singular limit of the latter, thus establishing 
a connection between two worlds often regarded as separate.

The singular limit is analysed by means of Gamma-convergence in the space of finite Borel measures endowed with the weak-$*$ topology.} 

\medskip

\section{Introduction}

\subsection{Chemical reaction as a diffusion process}
In a seminal paper in 1940, Hendrik Anthony Kramers described a number of approaches to the problem of calculating chemical reaction rates~\cite{Kramers40}. One of the limit cases in this paper is equivalent to the motion of a Brownian particle in a (chemical) potential landscape. In this description a reaction event is the escape of the particle from one energy well into another. 

This description is interesting for a number of reasons. It provides a connection between two processes, diffusion and reaction, which are often---especially at the macroscopic level---viewed as completely separate. It also provides a link between a macroscopic effect---chemical reaction---and a more microscopic, underlying motion, and in doing so, it highlights the fact that diffusion and reaction ultimately spring from the \emph{same} underlying motion. It finally also allows for explicit calculation of reaction rates in terms of properties of the energy landscape. 

In this paper we contribute to this discussion by studying the limit process of high activation energy in the unimolecular reaction 
$A\rightleftharpoons B$. As a first contribution, this provides a rigorous proof of the result that Kramers had derived formally. At the same time we extend his result to a Brownian motion in the product space spanned by both the chemical variable of Kramers and the variables corresponding to position in space, resulting in a limit system that models not only chemical reaction but also spatial diffusion---a simple reaction-diffusion system.
%

With this paper we have two aims. The first is to clarify the mathematical---rigorous---aspects of the formal results of~\cite{Kramers40}, and extend them to include spatial diffusion, and in this way to contribute to the upscaling of microscopic systems. The second is to
make a first step in the construction of a variational framework that
can describe the combination of general diffusive and chemically
reactive processes. 
From this point of view it would be interesting, for example, to place the limit system in the context of Wasserstein gradient flows (see also Section~\ref{sec:discussion}). Initiated by the work of Otto~\cite{JordanKinderlehrerOtto98,Otto01} and extended into many directions since, this framework provides an appealing variational structure for very general diffusion processes, but chemical reactions have so far resisted representation in the Wasserstein framework.

In this paper we only treat the simple equation
$A\rightleftharpoons B$, but we plan to extend the approach to other
systems in the future
(see also~\cite{MielkeInPrep}).

\subsection{The setup: enthalpy}

We consider the unimolecular reaction $A\rightleftharpoons B$. 
In chemical terms the $A$ and $B$ particles are two forms of the same molecule, such that the molecule can change from one form into the other. A typical example is a molecule with spatial asymmetry, which might exist in two distinct, mirror-image spatial configurations; another example is that of enzymes, for which the various spatial configurations also have different biological functions. 

\begin{remark}
Classical, continuum-level modelling of the system of $A$ and $B$ particles that diffuse and react (see e.g.~\cite{ErdiToth89,Aris99}) leads to the set of differential equations, where we write $A$ and $B$ for the concentrations of $A$ and $B$ particles:
\begin{subequations}
\label{pde:limit}
\begin{align}
\partial_t A - D\Delta A &= k(B-A)\\
\partial_t B - D\Delta B &= k(A-B).
\end{align}
\end{subequations}
(See Section~\ref{sec:discussion} for the equal reaction rates). 
This system will arise as the upscaling limit (see Theorem~\ref{th:1}) of the system that we now develop in detail. 
\end{remark}

We next assume that the observed forms $A$ and $B$ correspond to the wells of an appropriate energy function. Since it is common in the chemical literature to denote by `enthalpy difference' the release or uptake of heat as a particle $A$ is converted into a particle $B$, we shall adopt the same language and consider the $A$ and $B$ states to correspond to the wells of an enthalpy function $H$. 

While the domain of definition of $H$ should be high-dimensional, corresponding to the many degrees of freedom of the atoms of the molecule, we will here make the standard reduction to a one-dimensional dependence. The variable $\xi$ is assumed to parametrize an imaginary `optimal path' connecting the states $A$ to $B$, such that $\xi=-1$ corresponds to $A$ and $\xi=1$ to~$B$. Such a path should pass through the `mountain pass', the point which separates the basins of attraction of $A$ and $B$, and we arbitrarily choose that mountain pass to be at $\xi=0$, with $H(0)=1$. We also restrict $\xi$ to the interval $[-1,1]$, and we assume for simplicity that the wells are at equal depth, which we choose to be zero. 
A typical example of the function $H$ is showed in Figure~\ref{fig:H}.
\begin{figure}[ht]
\centering
\noindent
\psfig{figure=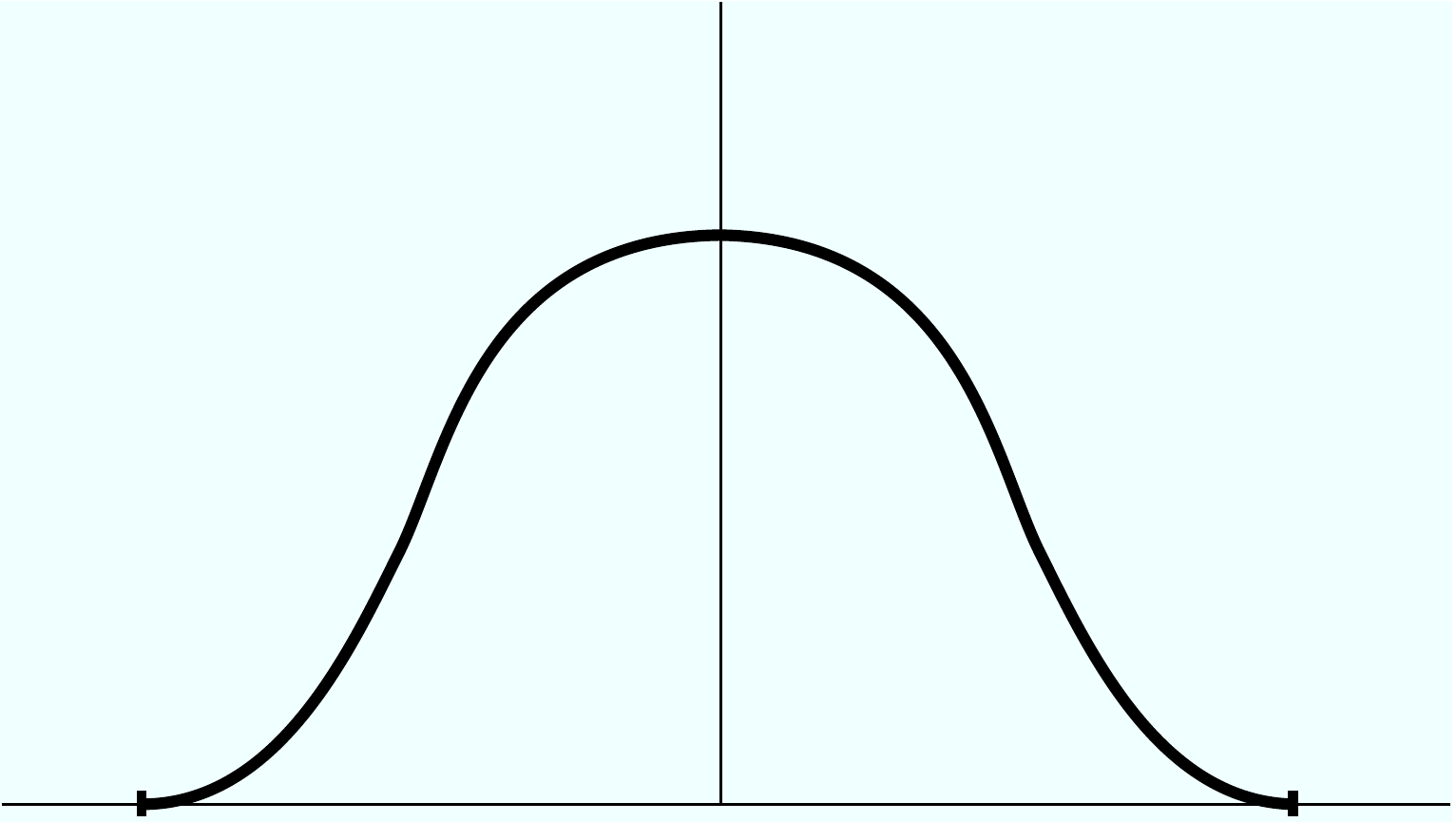,height=2cm}
\caption{A typical function $H$}
\label{fig:H}
\end{figure}

Specifically we make the following assumptions about $H$: $H\in C^\infty([-1,1])$, $H$ is even in $\xi$, maximal at $\xi=0$ with value $1$, and minimal at $\xi=\pm1$ with value $0$; $H(\xi)>0$ for any $-1<\xi<1$; $H'(\pm 1^\mp)=0$.  The assumption of equal depth for the two wells corresponds to an assumption about the rate constants of the two reactions; we comment on this in Section~\ref{sec:discussion}.

\subsection{Diffusion in the chemical landscape}

This newly introduced `chemical variable' $\xi$ should be interpreted as an internal degree of freedom of the particle, associated with internal changes in configuration. In the case of two alternative states of a molecule, $\xi$ parametrizes all the intermediate states along a connecting path.

In this view the total state of a particle consists of this chemical
state $\xi$ together with the spatial position of the particle,
represented by a $d$-dimensional spatial variable $x$
 in a Lipschitz, bounded, and open domain $\T\subset \R^d$,
so that the full state space for the particle is the closure $\cDom$
of 
\[
\Dom := \T\times (-1,1)
\qquad\text{with variables $(x,\xi)$}.
\]
Taking a probabilistic point of view, and following Kramers, the
motion of the particle will be described in terms of its probability
density $\rho\in \P(\cDom)$, in the sense that for Borel sets
$ X\subset \overline\T$ and $\Xi\subset [-1,1]$ the number
$\rho({X\times \Xi})$
is the probability of finding the particle at a position $x\in  X$ and with a `chemical state'
 $\xi\in \Xi$.

The particle is assumed to perform a Brownian motion
in $\Dom$, under the influence of the potential landscape described by $H$. 
This assumption corresponds to the `large-friction limit' discussed by Kramers.
The time evolution of the probability distribution $\rho$ then is
given by the
 Kramers-Smoluchowski
equation
\begin{equation}
\label{eq:pde-overline}
\partial_t \rho - \Delta_x \rho - 
  \tau\partial_\xi\bigl(\partial_{\xi}\rho + \rho\, \partial_\xi H\bigr) = 0
\qquad
\text{in $\D'(\Dom\times(0,\infty))$,}
\end{equation}
 with Neumann boundary conditions on the lateral boundary $\partial
  \Dom$.
The coefficient $\tau >0$ is introduced to parametrize the difference in scales for $x$ and $\xi$: since $x$ is a rescaled physical distance, and $\xi$ is a rescaled `chemical' distance, the units of length in the two variables are different, and the parameter $\tau$ can be interpreted as the factor that converts between the two scales. 
Below we shall make an explicit choice for $\tau$.

\subsection{The limit of high activation energy}

In the setup as described above, there is a continuum of states (i.e. $(-1,1)$) connecting the $A$ state to the $B$ state, and a statement of the type `the particle is in the $A$ state' is therefore not well defined. In order to make a connection with the macroscopic description `$A\rightleftharpoons B$', which presupposes a clear distinction between the two states, we take the limit of high activation energy, as follows.

We rescale the enthalpy $H$ with a small parameter $\e$, to make it $H(\xi)/\e$. (This is called `high activation energy' since $\max_\xi H(\xi)/\e = 1/\e$ is the height of the mountain that a particle has to climb in order to change states). 

This rescaling has various effects on the behaviour of solutions $\rho$ of~\pref{eq:pde-overline}. To illustrate one effect, let us consider the invariant measure $\gamma_\e$, the unique stationary solution in $\P(\cDom)$ of~\pref{eq:pde-overline}:
\begin{equation}
\gamma_\e =\lambdaO\otimes \tilde\gamma_\e,\quad
\lambdaO:=\frac1{\Ld(\T)}\Ld\restr\T,\quad
\tilde\gamma_\e =Z_\e^{-1} \rme^{-H/\e}\L^1\restr{[-1,1]}
\label{eq:28}
\end{equation}
(where $\L^1,\Ld$ are the $1$- and $d$-dimensional Lebesgue measures).
The constant $Z_\e$ is fixed by the requirement that $\gamma_\e(\cDom)=\tilde\gamma_\e([-1,1]) = 1$.

\begin{figure}[ht]
\centering
\noindent
\psfig{figure=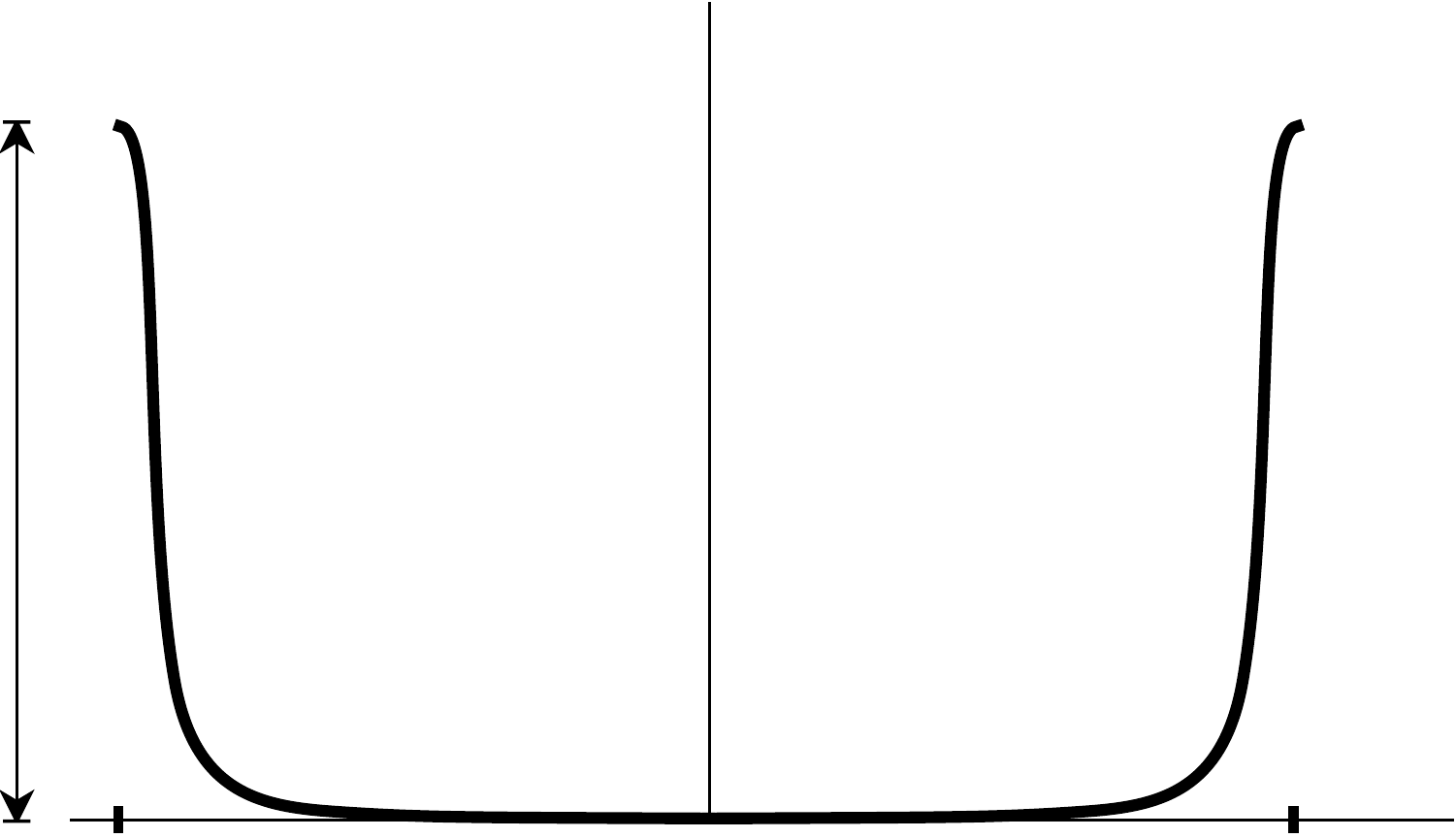,height=2.5cm}
\caption{The density $\tilde \gamma_\e$}
\label{fig:gamma-e}
\end{figure}

Since $H$ is strictly positive at any $-1<\xi<1$, the exponential
$\exp(-H(\xi)/\e)$ vanishes at all~$\xi$ except for $\xi=\pm 1$;
therefore the measure $\gamma_\e$ concentrates on the lines $\xi=-1$
and $\xi=1$, and
converges weakly-$*$ as $\e\to0$ to the limit measure $\gamma$ given by 
\begin{equation}
\gamma =
 \lambdaO\otimes \tilde\gamma,\quad
  \tilde\gamma:=\frac12\bigl(\delta_{-1}+\delta_1\bigr).
\label{eq:29}
\end{equation}
 Here weak-$*$ convergence is to be interpreted in the duality
  with continuous functions in $\cDom$ (thus considering
  $\P(\cDom)$ as a weakly-$*$ closed convex subset of
  the space $\M(\cDom)=\big(\Cb(\cDom)\big)'$ of
  signed Borel measures with finite total variation) i.e.
\[
 \lim_{\e\to0}\int_{\cDom}\phi(x,\xi)\,\d\gamma_\e=
\int_\cDom \phi(x,\xi) \, \d\gamma(x,\xi) = \frac12 \int_\T
\bigl(\phi(x,-1)+\phi(x,1)\bigr)\, \d \lambdaO(x),
\quad \text{for any }\phi\in \Cb(\cDom).
\]
\medskip

We should interpret the behaviour of $\gamma_\e$ as follows. In the limit $\e\to0$, the deep wells at $\xi=\pm 1$ force particles to stay increasingly close to the bottom of the wells. 
However, at any given $\e>0$, there is a positive probability that a particle switches from one well to the other in any given period of time. The rate at which this happens is governed by the local structure of $H$ near $\xi=\pm 1$ and near $\xi=0$, and becomes very small---of order $\e^{-1}\exp(-1/\e)$, as we shall see below. 

In the limit $\e=0$, the behaviour of particles in the $\xi$-direction is no longer recognizable as diffusional in nature. In the $\xi$-direction a particle can only be in one of two states $\xi=\pm1$, which we therefore interpret as the $A$ and $B$ states. Of the diffusional movement in the $\xi$-direction only a jump process remains, in which a particle at $\xi=-1$ jumps with a certain rate to position $\xi=1$, or vice versa. 

\subsection{Spatiochemical rescaling}\label{sec:rescaling}

Since the jumping (chemical reaction) rate at finite $\e>0$ is of order $\e^{-1}\exp(-1/\e)$, the limiting reaction rate will be zero unless we rescale the system appropriately. This requires us to speed up time by a factor of $\e\exp(1/\e)$. At the same time, the diffusion rate in the $x$-direction remains of order $1$ as $\e\to0$, and the rescaling should preserve this. In order to obtain a limit in which both diffusion in $x$ and chemical reaction in $\xi$ enter at rates that are of order $1$, we use the freedom of choosing the parameter $\tau$ that we introduced above.

We therefore choose $\tau$ equal to 
\begin{equation}
\label{def:tau_e}
\tau_\e := \e\exp(1/\e),
\end{equation}
and we then find the differential equation
\begin{equation}
\label{eq:pde-epsilon2}
\partial_t \rhoe - \Delta_x \rhoe - 
  \tau_\e \partial_\xi(\partial_\xi\rhoe +\tfrac1\e\rhoe\,\partial_\xi H) = 0\quad
\text{in $\D'(\Dom\times(0,\infty))$,}
\end{equation}
which clearly highlights the different treatment of $x$ and $\xi$: the diffusion in $x$ is independent of $\tau_\e$ while  the diffusion and convection in the $\xi$-variable are accelerated by a factor $\tau_\e$.

\subsection{Switching to the density variable}
\label{subsec:density_var}

As is already suggested by the behaviour of the invariant measure
$\gamma_\e$, the solution $\rhoe$ will become strongly
concentrated at the extremities $\{\pm 1\}$ of the $\xi$-domain
$(-1,1)$.
 This is the reason why it is useful to interpret $\rhoe$ as a family
  $\rhoe(t,\cdot)$ of 
  \emph{time-dependent measures},
  instead of functions. It turns out that the densities $\ue(t,\cdot)$
  \[
  \ue(t,\cdot) := \frac{\d\rho(t,\cdot)}{\d\gamma_\e}
  \]
  of
  $\rhoe(t,\cdot)$ with respect to $\gamma_\e$ also play a crucial
  role and it is often 
  convenient to have both representations at our disposal, freely switching
  between them.
In terms of the variable $\ue$ equation~\pref{eq:pde-epsilon2} becomes
\begin{equation}
  \label{eq:pde-u}
  \partial_t \ue - \Delta_x \ue  - \tau_\e(\partial_{\xi\xi}^2\,u_\e - \tfrac1\e\partial_\xi H\partial_\xi \ue)=0
  \quad\text{in }(0,+\infty)\times\Dom,
\end{equation}
 supplemented
with the boundary conditions
\begin{equation}
\label{bc:xi}
\partial_\xi u_\e(t, x,\pm 1) = 0
\quad\text{for all }x\in\T,\qquad
 \nnabla_x u_\e(t, x,\xi)\cdot \mathbf n =0\quad \text{on
  }\partial\T\times [-1,1],\qquad t>0.
\end{equation}
We choose an initial condition
\begin{equation}
\label{ic:u}
u_\e(0,x,\xi) = u^0_\e(x,\xi),
\qquad\text{for all $(x,\xi)\in \Dom$},\quad
 \text{with}\quad \rho^0_\e=u^0_\e\gamma_\e\in \P(\cDom).
\end{equation}

\medskip
Let us briefly say something about the functional-analytic setting. 
It is well known (see e.g.~\cite{DaPRatoLunardi07}) that the operator
$A_\e:=-\Delta_x-\tau_\e \partial^2_{\xi\xi}+(\tau_\e/\e)
 H' \,\partial_\xi$ with Neumann boundary conditions~\eqref{bc:xi} has a
self-adjoint realization in the space $H_\e := L^2(\Dom;\gamma_\e)$. Therefore the weak form of equation~\eqref{eq:pde-u} can be written as 
\begin{equation}
\label{eq:58}
b_\e(\partial_t u(t),v) + a_\e(u(t),v) =0
\qquad \text{for all $v\in V_\e$},
\end{equation}
where the bilinear forms $a_\e$ and $b_\e$ are defined by
\[
b_\e:H_\e\times H_\e\to \R, \qquad
b_\e(u,v) := \int_D u\,v\,\d \gamma_\e,
\]
and
\begin{align*}
&V_\e := W^{1,2}(\Dom;\gamma_\e):=\Big\{u\in L^2(\Dom\GG,\gamma_\e)\cap W^{1,1}_{\rm loc}(\Dom):
  \int_\Dom |\nnabla_{x,\xi}u|^2\,\d\gamma_\e<+\infty\Big\},\\
&a_\e:V_\e\times V_\e\to\R, \qquad a_\e(u,v):=\int_\Dom A_\e u\,v\,\d\gamma_\e=
  \int_\Dom
  \Big(\nnabla_x u\nnabla_x v+\tau_\e \partial_\xi
  u_\e\partial_\xi\, v\Big)\,\d\gamma_\e.
\end{align*}
{Since $V_\e$ is densely and continuously imbedded in $H_\e$},
standard results on variational evolution equations in an Hilbert 
triplet (see e.g.\ \cite{LionsMagenes72,Brezis83}) and their regularizing effects show that a unique solution exists in $C([0,\infty);H_\e) \cap C^\infty((0,\infty);V_\e)$ for
every initial datum $u^0_\e\in H_\e$.

\subsection{Main result I: weak convergence of $\rho_\e$ and $u_\e$}

The following theorem is the first main result of this paper.
It states that for every time $t\ge 0$ the 
measures $\rho_\e(t)$ solutions of \eqref{eq:pde-epsilon2} weakly-$*$ converge
to a limiting measure $\rho(t)$ in $\P(\cDom)$, whose density $u(t)=
\frac{\d\rho(t)}{\d\gamma}$ is the solution of the limit system~\eqref{pde:limit}. 
Note that for a function $u\in L^2(\cDom,\gamma)$ the traces $u^\pm=u(\cdot,\pm1)\in L^2(\Omega)$
are well defined (in fact, the map $u\mapsto (u^-,u^+)$ is an isomorphism between
  $L^2(\cDom,\gamma)$ and $L^2(\Omega,\tfrac 12\lambdaO;\R^2)).$

We state our result in a general form, which holds even for signed
measures in $\M(\cDom)$.
\begin{theorem}
\label{th:1}
Let $\rho_\e=u_\e\,\gamma_\e\in \mathrm C^0([0,+\infty);\M(\cDom))$ be the solution
of~(\ref{eq:pde-epsilon2}--\ref{ic:u}) with initial datum~$\rho_\e^0$.
If
\begin{equation}
  \label{eq:52}
  \sup_{\e>0}\int_\Dom |u^0_\e|^2\,\d\gamma_\e<+\infty
\end{equation}
and
\begin{equation}
  \label{eq:53}
  \text{$\rho_\e^0$ weakly-$*$ converges to
    $\rho^0=u^0\gamma=\frac 12u^{0-}\,\lambdaO\otimes \delta_{-1}+
    \frac12u^{0+}\,\lambdaO\otimes\delta_{+1}$ as $\e\downarrow0$,}
\end{equation}
then $u^0\in L^2(\cDom\GG,\gamma)$, $u^{0,\pm}\in L^2(\Omega)$, and
for every $t\ge0$ the solution $\rho_\e(t)$ weakly-$*$ converge to
\begin{equation}
  \label{eq:51}
  \rho(t)=u(t)\,\gamma=\frac12 u^-(t)\,\lambdaO\otimes \delta_{-1}+
  \frac12u^+(t)\,\lambdaO\otimes\delta_{+1},
\end{equation}
whose densities $u^\pm$ belong to $C^0([0,+\infty);L^2(\Omega))\cap
C^1((0,+\infty);W^{1,2}(\Omega))$ 
and solve the system
\begin{subequations}
\label{pb:limit}
\begin{alignat}2
  \label{eq:32}&\partial_t u^+ - \Delta_x u^+ = k(u^--u^+)
  \qquad&&\text{in  }\Omega\times(0,+\infty)\\
  \label{eq:33}&\partial_t u^- - \Delta_x u^- = k(u^+-u^-)\qquad&&\text{in  }\Omega\times(0,+\infty)\\
  \label{eq:34}&u^\pm(0) = u^{0,\pm}&&\text{in }\Omega.
\end{alignat}
\end{subequations}
The positive constant $k$ in \emph{(\ref{eq:32},b)} can be characterized as the
asymptotic minimal transition cost
\begin{equation}
\label{def:k}
k=\frac1\pi \sqrt{|H''(0)|H''(1)}=
\lim_{\e\downarrow0}\min\Big\{
\tau_\e\int_{-1}^1
\big(\varphi'(\xi)\big)^2\,\d\tilde\gamma_\e: \varphi\in
W^{1,2}(-1,1),\ \varphi(\pm 1)=\pm \tfrac12\Big\}.
\end{equation}
\end{theorem}
\begin{remark}[The variational structure of the limit problem]
  The ``$\e=0$'' limit problem~(\ref{eq:32}-\ref{eq:34}) admits the same variational
  formulation of  the ``$\e>0$'' problem we introduced in
  Section~\ref{subsec:density_var}.
  Recall that $\gamma$ is the measure defined in~\eqref{eq:29} as
  the weak limit of $\gamma_\e$; we set $H:=L^2(\cDom,\gamma)$,
  and for every $\rho=u\gamma$ with $u\in H$ we set $u^{\pm}(x):=
  u(x,\pm 1)\in L^2(\Omega,\lambdaO)$.
  We define 
  \begin{equation}
    \label{eq:30}
    b(u,v):=\int_\sDom u(x,\xi)v(x,\xi)\,\d\gamma(x,\xi)=
    \frac 12\int_\Omega \Big(u^+v^++u^-v^-\Big)\,\d \lambdaO.
  \end{equation}
  Similarly, we set $V:=\big\{u\in H: u^\pm\in
  W^{1,2}(\Omega)\big\}$, which is continuously and densely imbedded in $H$, and
  \begin{equation}
    \label{eq:31} 
    a(u,v):=
    \frac12\int_\Omega \Big(\nnabla_x u^+\nnabla_x v^++\nnabla_x
    u^-\nnabla_x v^-+k\big(u^+-u^-\big)(v^+-v^-)\Big)\,\d \lambdaO.
  \end{equation}
  Then
  the system (\ref{pb:limit}a,b,c) can be formulated
  as
  \begin{equation}
    \label{eq:59}
    b(\partial_t u(t),v) + a(u(t),v)=0\quad
    \text{for every $t>0$ and $v\in V$},
  \end{equation}
  which has the same structure as \eqref{eq:58}.
\end{remark}

\subsection{Main result II: a stronger convergence of $u_\e$}
\label{subsec:mainresultII}

Weak-$*$ convergence in the sense of measures is a natural choice in
order to describe the limit of $\rho_\e$, since the densities $u_\e$
and the limit density $u=(u^+,u^-)$ are defined on different domains
{with respect to different reference measures}.
Nonetheless it is possible to {consider a} stronger
convergence
which better characterizes the limit, and to prove that it is satisfied by the solutions of our problem. 

This stronger notion is modeled on Hilbert spaces  (or, more
generally, on Banach spaces with a locally uniformly convex norm),
where strong convergence is equivalent to weak convergence together with the convergence of the norms:
\begin{equation}
\label{equiv:strong-norms}
x_n \to x \quad \Longleftrightarrow \quad x_n \weakto x\quad  \mbox{and} \quad \|x_n\|\to \|x\|.
\end{equation}
In this spirit, the next result states that under the additional
request of ``strong'' convergence of the initial data $u_\e^0$, we
have ``strong'' convergence of the densities $u_\e$;
we refer to
  \cite{Reshetnjak68,Hutchinson86} (see also
  \cite[Sec. 5.4]{AmbrosioGigliSavare05})
  for further references in a measure-theoretic setting.

\medskip

\begin{theorem}
\label{th:2}
Let $\rho_\e,$ $\rho_\e^0$ be as in Theorem \ref{th:1}. If moreover
\begin{equation}
  \label{eq:54}
  \lim_{\e\downarrow0}b_\e(u_\e^0,u_\e^0)
  =b(u^0,u^0),
\end{equation}
then for every $t>0$ we have
\begin{equation}
  \label{eq:55}
  \lim_{\e\downarrow0}b_\e(u_\e(t),u_\e(t))= b(u(t),u(t))
\end{equation}
and
\begin{equation}
  \label{eq:56}
  \lim_{\e\downarrow0}a_\e(u_\e(t),u_\e(t))=
  a(u(t),u(t)).
\end{equation}
\end{theorem}
  Applying, e.g., \cite[Theorem 5.4.4]{AmbrosioGigliSavare05} we can
  immediately deduce the following result, which clarifies 
  the strengthened form of convergence that we are
  considering here. This convergence is strong enough to allow us to pass to the limit in 
  nonlinear functions of $u_\e$:
  \begin{corollary}
    Under the same assumptions as in Theorem \ref{th:2} we
    have 
    \begin{align}
      \label{eq:1}
      \lim_{\e\downarrow0}\int_\Dom f(x,&\xi,u_\e(x,\xi,t))\,\d\gamma_\e(x,\xi)=
        \int_\cDom f(x,\xi,u(x,\xi,t))\,\d\gamma(x,\xi)
        \\\notag&= \frac12\int_\Omega
        \Big(f(x,-1,u^-(x,t))+f(x,1,u^+(x,t))\Big)\,\d\lambdaO(x)
        \quad\text{for every $t>0$},
      \end{align}
    where $f:\cDom\times \R\to \R$ is an arbitrary continuous function
    satisfying 
    the quadratic growth condition
    \begin{displaymath}
      |f(x,\xi,r)|\le A+Br^2\quad\text{for every }(x,\xi)\in \cDom,\ r\in \R
    \end{displaymath}
    for suitable nonnegative constants $A,B\in \R$.
  \end{corollary}

\subsection{Structure of the proof}
\label{sec:structure}
Let us briefly explain the structure of the proof of Theorems~\ref{th:1} and~\ref{th:2}. This will also clarify the term $\Gamma$-convergence in the title, and highlight the potential of the method for wider application.

The analogy between~\pref{eq:58} and~\pref{eq:59} suggests to pass to
the limit in these weak formulations, or even better, in their
  equivalent integrated forms
  \begin{equation}
    \label{eq:4}
    b_\e(u_\e(t),v_\e)+\int_0^t a_\e(u_\e(t),v_\e)\,\d
    t=b(u_\e^0,v_\e),\quad
    b(u(t),v)+\int_0^t a(u(t),v)\,\d t=b(u^0,v).
  \end{equation}
Applying standard regularization estimates for the solutions to
\eqref{eq:58} and a weak coercivity property of $b_\e$, it is not
difficult to prove that $u_\e(t)$ ``weakly'' converges to $u(t)$ for
every $t>0$, i.e. 
\begin{displaymath}
  \rho_\e(t)=u_\e(t)\gamma_\e\weakto\rho(t)=u(t)\gamma\quad
  \text{weakly-$*$ in $\M(\cDom)$.}
\end{displaymath}
The concept of weak convergence of densities that we are using
here is thus the same as in Theorem~\ref{th:1}, i.e. weak-$*$
convergence of the corresponding measures in $\M(\cDom)$.

In order to pass to the limit in \eqref{eq:4} the central property is the 
following
\emph{weak-strong} convergence principle:
\begin{quote}
  For every $v\in V$ there exists $v_\e\in V_\e$ with
  $v_\e\weakto v$ as $\e\to0$ such that
  for every $u_\e\weakto u$
\[
b_\e(u_\e,v_\e) \to b(u,v)
\qquad\text{and}\qquad
a_\e(u_\e,v_\e) \to a(u,v).
\]
\end{quote}
Note that the previous property implies in particular that
  recovery family $v_\e$ converges ``strongly'' to $v$, according to the notion
  considered
  by Theorem~\ref{th:2}, i.e. $v_\e\to v$ iff $v_\e\weakto v$ with  both $b_\e(v_\e,v_\e) \to b(v,v)$ and $a_\e(v_\e,v_\e) \to a(v,v)$.
%
Corollary~\ref{cor:Gamma} shows that this weak-strong convergence
property can be derived from
$\Gamma$-convergence in the ``weak'' topology of the family of quadratic forms 
\begin{equation}
  q_\e^\kappa(u):=b_\e(u,u)+\kappa\, a_\e(u,u)\quad
    \text{to}\quad
      q^\kappa(u):=b(u,u)+\kappa\,a(u,u)\quad\text{for }\kappa>0.\qquad
\label{eq:5}
\end{equation}
In order to formulate this property in the standard framework of
  $\Gamma$-convergence we will extend $a_\e$ and $b_\e$ to
  lower semi-continuous quadratic
  functionals
  (possibly assuming the value $+\infty$) in the
  space $\M(\cDom)$, following the approach of
  \cite[Chap.~11-13]{DalMaso93}.
  While the $\Gamma$-convergence of $b_\e$ is a direct
    consequence of the weak convergence of $\gamma_\e$ to $\gamma$,
  the convergence of $a_\e$ is more subtle. The convergence of $a_\e$ and the structure of the limit depends critically on the choice of $\tau_\e$~(defined in~\eqref{def:tau_e}): as we show in Section~\ref{subsec:minimal_cost}, the scaling of $\tau_\e$ in terms of $\e$ is chosen exactly such that the strength of the `connection' between $\xi=-1$ and $\xi=1$ is of order $O(1)$ as $\e\to0$. 

\medskip
The link between $\Gamma$-convergence and stability of evolution
  problems of parabolic type is well known  when $b_\e=b$ is a fixed and coercive bilinear
form (see, e.g., \cite[Chap.~3.9.2]{Attouch84}) and can therefore be
considered as the scalar product of the Hilbert space $H_\e\equiv H$.
In this case the embedding of the problems in a bigger topological
vector space (the role played by $\M(\cDom)$ in our situation) is no
more needed, and one can deal with the weak and strong topology of
$H$,
obtaining the following equivalent characterizations
(see e.g.~\cite[Th. 3.16]{Brezis73} and~\cite[Th. 13.6]{DalMaso93}):
\begin{enumerate}
\item Pointwise (strong) convergence in $H$ of the solutions of the evolution
  problems; 
\item Pointwise convergence in $H$ of the resolvents of the linear
    operators associated to the bilinear forms $a_\e$;
\item Mosco-convergence in $H$ of the quadratic forms associated
    to $a_\e$;
\item $\Gamma$-convergence in the weak topology of $H$ of the
  quadratic forms $b+\kappa\, a_\e$ to $b+\kappa\, a$
  for every $\kappa>0$.
\end{enumerate}

In the present case, where $b_\e$ does depend on $\e$,
$\Gamma$-convergence of the extended quadratic forms $b_\e +
\kappa\, a_\e$
with respect to the weak-$*$ topology of $\M(\cDom)$
is thus a natural extension of the latter condition;
Theorem~\ref{thm:Gamma} can be interpreted as essentially proving
a slightly stronger version of this property.
Starting from this $\Gamma$-convergence result, we will derive
  the convergence of the evolution problems by a simple and general
  argument, which we will present in Section \ref{sec:proof}.

\subsection{Discussion}
\label{sec:discussion}

The result of Theorem~\ref{th:1} is amongst other things a rigorous version of the result of Kramers~\cite{Kramers40} that was mentioned in the introduction. It shows that the simple reaction-diffusion system~\pref{pb:limit} can indeed be viewed as an upscaled version of a diffusion problem in an augmented phase space; or, equivalently, as an upscaled version of the movement of  a Brownian particle in the same augmented phase space. 

At the same time it generalizes the work of Kramers by adding the spatial dimension, resulting in a limit system which---for this choice of $\tau_\e$, see below for more on this choice---captures both reaction and diffusion effects. 

\medskip
\emph{Measures versus densities. } It is interesting to note the roles of the measures $\rho_\e,\rho$ and their densities $u_\e,u$ with respect to $\gamma_\e,\gamma$.
 The variational formulation of the equations are done in terms of the densities $u_\e,u$
  but the limit procedure is better understood in terms of the measures $\rho_\e,\rho$,
  since a weak-$*$ convergence is involved. This also allows for a unification of two problems with
  a  different structure (a Fokker-Planck equation for $u_\e$ and a reaction-diffusion system
  for the couple $u^-,u^+$.)
\medskip

\emph{Gradient flows.} 
The weak formulation \pref{eq:58} shows also that a solution $u_\e$
can be interpreted as a gradient flow of the quadratic energy
$\frac 12 a_\e(u,u)$ with respect to the $L^2(\Dom\GG, \gamma_\e)$ distance. Another gradient flow structure for the solutions of the same problem could be obtained by a different choice of energy functional and distance: for example, as proved in~\cite{JordanKinderlehrerOtto98}, Fokker-Planck equations
like \eqref{eq:pde-epsilon2} can be interpreted also as the gradient flow
of the  relative entropy functional
\begin{align}
  \label{eq:14}
  \mathcal H(\rho|\gamma_\e) := &
     \int_\sDom \frac{\d
      \rho}{\d\gamma_\e}\log\Big(\frac{\d\rho}{\d\gamma_\e}\Big)\,\d\gamma_\e
\end{align}
in the space $\P(\cDom)$ of probability measures endowed with
the so-called $L^2$-Wasserstein distance
(see e.g.\ \cite{AmbrosioGigliSavare05}).
Other recent work~\cite{AdamsDirrPeletierZimmerInPrep} suggests that the Wasserstein setting can be the most natural for understanding diffusion as a limit of the motion of Brownian particles, but in this case it is not obvious how to interpret the limit system in the framework of gradient flows on probability measures, and how to obtain it in the limit as $\e \to 0$. 

In a forthcoming paper we investigate a new distance for the limit problem, modeled on the reaction-diffusion term, and we study how the limit couple of energy and dissipation can be obtained as a $\Gamma$-limit. 
\medskip

\emph{The choice of $\tau_\e$.} 
In this paper the time scale $\tau_\e$ is chosen to be equal to $\e \exp(1/\e)$, and it is a natural question to ask about the limit behaviour for different choices of $\tau_\e$. If the scaling is chosen differently---i.e.\ if $\tau_\e \e^{-1}\exp(-1/\e)$ converges to $0$ or $\infty$---then completely different limit systems are obtained:
\begin{itemize}
\item If $\tau_\e \ll \e \exp(1/\e)$, then the reaction is not accelerated sufficiently as $\e\to0$, and the limit system will contain only diffusion (i.e.\ $k=0$ in~\pref{pb:limit});
\item If $\tau_\e \gg \e \exp(1/\e)$,  on the other hand, then the reaction is made faster and faster as $\e\to0$, resulting in a limit system in which the chemical reaction $A\leftrightharpoons B$ is in continuous equilibrium. Because of this, both $A$ and $B$ have the same concentration $u$, and $u$ solves the diffusion problem
\begin{alignat*}2
&\partial_t u = \Delta u,  \qquad&&\text{for }x\in \T, \ t>0\\
&u(0,x) = \frac12\bigl(u^{0,+}(x)+ u^{0,-}(x)\bigr)\qquad &&\text{for }x\in \T.
\end{alignat*}
Note the instantaneous equilibration of the initial data in this system.
\end{itemize}

While the scaling in terms of $\e$ of $\tau_\e$ can not be chosen differently without obtaining structurally different limit systems, there is still a choice in the prefactor. For $\tau_\e := \tilde \tau \e e^{1/\e}$ with $\tilde \tau>0$ fixed, the prefactor $\tilde \tau$ will appear in the definition~\pref{def:k} of $k$. 

There is a also a modelling aspect to the choice of $\tau$. In this paper we use no knowledge about the value of $\tau$ in the diffusion system at finite $\e$; the choice $\tau=\tau_\e$ is motivated by the wish to have a limit system that contains both diffusive and reactive terms. If one has additional information about the mobility of the system in the $x$- and $\xi$-directions, then the value of $\tau$ will follow from this.

%
%

\medskip

\emph{Equal rate constants.} The assumption of equal depth of the two minima of $H$ corresponds to the assumption (or, depending on one's point of view, the result) that the rate constant $k$ in~\pref{pb:limit} is the same for the two reactions $A\to B$ and $B\to A$. The general case requires a slightly different choice for $H$, as follows.

Let the original macroscopic equations for the evolution of $A$ and $B$ (in terms of densities that we also denote $A$ and $B$) be
\begin{subequations}
\label{pb:limit_skew}
\begin{align}
\partial_t A - \Delta A &= k^- B - k^+ A\\
\partial_t B - \Delta B &= k^+ A - k^- B.
\end{align}
\end{subequations}
Choose a fixed function $H_0\in C^\infty([-1,1])$ such that $H_0'(\pm1)=0$ and $H_0(1)-H_0(-1)=\log k^- - \log k^+$. We then construct the enthalpy $H_\e$ by setting
\[
H_\e := H_0 + \frac1\e H,
\]
where $H$ is the same enthalpy function as above. The same proof as for the equal-well case then gives convergence of the finite-$\e$ problems to~\pref{pb:limit_skew}.
%
%

\medskip

\emph{Equal diffusion constants.} It is possible to change the setup such that the limiting system 
has different diffusion rate in $A$ and $B$. We first write equation~\pref{eq:pde-epsilon2} as
\[
\partial_t \rho - \div D_\e F_\e = 0,
\]
where the mobility matrix $D_\e\in\R^{(d+1)\times(d+1)}$ and the flux $F_\e$ are given by
\[
D_\e = \begin{pmatrix}
\mathrm{I} & 0\\
0 & \tau_\e
\end{pmatrix}
\qquad\text{and}\qquad
F_\e = F_\e(\rho) = \begin{pmatrix} 
\nabla u\\
\nabla \rho + \rho\nabla H
\end{pmatrix}
\]
By replacing the identity matrix block $I$ in $D_\e$ by a block of the form $a(\xi)\,\mathrm I$ the $x$-directional diffusion can be modified as a function of $\xi$. This translates into two different diffusion coefficients for $A$ and $B$.

\medskip

\emph{The function $H$.} The limit result of Theorem~\ref{th:1} shows that only a small amount of information about the function $H$ propagates into the limit problem: specifically, the local second-order structure of $H$ around the wells and around the mountain-pass point. 

One other aspect of the structure of $H$ is hidden: the fact that we rescaled the $\xi$ variable by a factor of $\sqrt {\tau_\e}$ can also be interpreted as a property of $H$, since the effective distance between the two wells, as measured against the intrinsic distance associated with the Brownian motion, is equal to $2\sqrt{\tau_\e}$ after rescaling. 

We also assumed in this paper that $H$ has only `half' wells, in the sense that $H$ is defined on $[-1,1]$ instead of $\R$. This was for practical convenience, and one can do essentially the same analysis for a function $H$ that is defined on $\R$. In this case one will regain a slightly different value of $k$, namely $k=\sqrt{|H''(0)|H''(1)}/2\pi$. (For this reason this is also the value found by Kramers~\cite[equation (17)]{Kramers40}).

\medskip
\emph{Single particles versus multiple particles, and concentrations versus probabilities. }
The description of this paper of the system in terms of a  probability measure $\rho$ on $\Dom$ is the description of the probability of a \emph{single} particle. This implies that the limit object $(u^-,u^+)$ should be interpreted as the density (with respect to $\gamma$) of a limiting probability measure, again describing a single particle. 

This is at odds with common continuum modelling philosophy, where the main objects are concentrations (mass or volume) that represent a large number of particles; in this philosophy the solution $(u^-,u^+)$ of~\pref{pb:limit} should be viewed as such a concentration, which is to say as the projection onto $x$-space of a \emph{joint probability distribution of a large number of particles}. 

For the simple reaction $A\leftrightharpoons B$ these two interpretations are actually equivalent. This arises from the fact that $A\to B$ reaction events in each of the particles are independent of each other; therefore the joint distribution of a large number $N$ of particles factorizes into a product of  $N$ copies of the distribution of a single particle. For the case of this paper, therefore, the distinction between these two views is not important.

\medskip
\emph{More general reactions. }
The remark above implies that the situation will be different for systems where reaction events cause differences in distributions between the particles, such as the reaction $A+B\leftrightarrows C$. This can be recognized as follows: a particle $A$ that has just separated from a $B$ particle (in a reaction event of the form $C\to A+B$) has a position that is highly correlated with the corresponding $B$ particle, while this is not the case for all the other $A$ particles. Therefore the $A$ particles will not have the same distribution. The best one can hope for is that in the limit of a large number of particles the distribution becomes the same in some weak way. This is one of the major obstacles in developing a similar connection as in this paper for more complex reaction equations.

\drop{
\subsection{Chemical considerations}

Both this paper and that of Kramers~\cite{Kramers40} operate in a classical, i.e.\ non-quantum, world view, in which the position of the system can be considered to be a point instead of a probability distribution over points. The formulation that lies at the basis of this paper---chemical reaction as the movement of a Brownian particle in an appropriate space---then arises as follows. The description that follows is a summary of the derivation in~\cite{Kramers40}.

First one makes the assumption that the system (the molecule under consideration) is under a continuous barrage of interactions with other molecules, and that this barrage can be represented by a random fluctuating force with mean zero. This leads to a movement of the particle in an appropriate phase space. At this stage the phase space is a combination of the translational and the internal chemical degrees of freedom (similar to $\Dom$ in this paper) combined with the velocities in the same degrees of freedom. 

Next  one assumes that a time scale $\overline t$ exists with the following property: on one hand the velocity changes little over times of length $\overline t$, and on the other hand the correlation between the forces at times $t$ and $t+\overline t$ is negligible. The result is an extension of the Ornstein-Uhlenbeck stochastic process in which the deterministic force on the particle depends not only on the velocity but also on the position of the molecule. Finally one makes an assumption that the random forces are large with respect to the deterministic forces. This results in a Maxwellian velocity distribution, and causes the degrees of freedom to reduce to only the position in phase space.

\medskip
The original work of Kramers was aimed at the calculation of reaction rate constants (i.e.\ the coefficient $k$) by expressing it in terms of properties of $H$ as in~\pref{def:k}. This only translates the problem, since often little is known about the enthalpy function $H$. 
}
  \subsection{Plan of the paper}
  One of the main difficulties in the proof of Theorem~\ref{th:1}, namely
the singular behaviour given by the concentration of the invariant
measure $\gamma_\e$ onto the two lines at $\xi=\pm1$,
can be overcome by working in the underlying space of (signed or
probability) measures in $\cDom$. This point of view is introduced in Section~\ref{sec:variational}.  Section \ref{sec:Gamma} contains the basic $\Gamma$-convergence
  results (Theorem \ref{thm:Gamma}) and the proof of Theorem \ref{th:1} and of Theorem \ref{th:2}.
The argument showing the link between $\Gamma$-convergence of the
  quadratic forms $a_\e,b_\e$ and the convergence of the solutions to
  the evolution problems (see the comments in section
  \ref{sec:structure}) is presented in Section \ref{sec:proof} in a
  general form, which can 
  can be easily applied to other situations.



\section{Formulation of the evolution equations in measure spaces}
\label{sec:variational}
\subsection*{The Kramers-Smoluchowski equation}
We first summarize the functional framework introduced above. 
Let us denote by $(\cdot,\cdot)_\e$
the scalar product in $\R^d\times \R$ defined by
\begin{equation}
  \label{eq:19}
  (\xx,\yy)_{\e}:=x\cdot y+\tau_\e \, \xi\,\eta,\quad
  \text{for every }\xx=(x,\xi),\ \yy=(y,\eta)\in \R^d\times \R,
\end{equation}
with the corresponding norm $\|\cdot\|_\e$. We introduced two Hilbert spaces
\[
H_\e:=L^2(\Dom,\gamma_\e) 
\qquad\text{and}\qquad
V_\e=W^{1,2}(\Dom,\gamma_\e),
\]
and
the bilinear forms
\begin{align}
  \label{eq:18}
  b_\e(u,v):=&\int_\sDom u\, v\,\d\gamma_\e\quad&&\text{for every }u,v\in H_\e,\\
  \label{eq:20}
  a_\e(u,v):=&\int_\sDom (\nnabla_{x,\xi}
  u,\nnabla_{x,\xi}v)_{\e}\,\d\gamma_\e&&\text{for every }u,v\in V_\e,
\end{align}
with which \eqref{eq:pde-u} has the variational
formulation 
\begin{equation}
  \label{eq:22}
  b_\e(\partial_t u_\e,v)+a_\e(u_\e,v)=0\quad \text{for every }v\in
  V_\e,\ t>0;\quad
  u_\e(0,\cdot)=u_\e^0.
\end{equation}

The main technical difficulty in studying the limit behaviour of
\eqref{eq:22} as $\e\downarrow0$ consists of the
$\e$-dependence of the functional spaces $H_\e,V_\e$. Since
for our approach it is
crucial to work in a fixed ambient space, we embed the solutions of
\eqref{eq:22}
in the space of finite Borel measures $\M(\cDom)$ by associating to $u_\e$ the
measure
$\rho_\e:=u_\e\gamma_\e$.
We thus introduce the quadratic forms
\begin{align}
  \label{eq:23}
  \b_\e(\rho)&:=b_\e(u,u)&&\text{if $\rho\ll\gamma_\e$ and $u=\frac{\d\rho}{\d\gamma_\e}\in H_\e,$}\\
  \label{eq:24}
  \a_\e(\rho)&:=a_\e(u,u)&&\text{if $\rho\ll\gamma_\e$ and
    $u=\frac{\d\rho}{\d\gamma_\e}\in V_\e,$}
\end{align}
trivially extended to $+\infty$ when $\rho$ is not absolutely
continuous
with respect to $\gamma_\e$ or its density $u$ does not belong to
$H_\e$ or $V_\e$ respectively. Denoting by $\mathsf{Dom}(\a_\e)$ and
$\mathsf{Dom}(\b_\e)$ their proper domains, we still denote by
$\a_\e(\cdot,\cdot)$ and
$\b_\e(\cdot,\cdot)$ the corresponding bilinear forms defined on 
$\mathsf{Dom}(\a_\e)$ and
$\mathsf{Dom}(\b_\e)$ respectively.
Setting $\rho_\e:=u_\e\gamma_\e$, $\sigma:=v\gamma_\e$,
\eqref{eq:22} is equivalent to the integrated form
\begin{equation}
  \label{eq:25}
  \b_\e(\rho_\e(t),\sigma)+\int_0^t
  \a_\e(\rho_\e(r),\sigma)\,\d r=\b_\e(\rho_\e^0,\sigma)\quad
  \text{for every }\sigma\in \mathsf{Dom}(\a_\e).
\end{equation}
We also recall the standard estimates
\begin{alignat}2
  \label{eq:26}
  &\frac 12\b_\e(\rho_\e(t))+\int_0^t\a_\e(\rho_\e(r))\,\d r=\frac
  12\b_\e(\rho_\e^0)&&\quad\text{for every }t\ge0,\\
    \label{id:a-intb}
    &t\,\a_\e(\rho_\e(t))+2\int_0^t r\b_\e(\partial_t \rho_\e(r))\,\d
    r=\int_0^t \a_\e(\rho_\e(r))\,\d r
  &&\quad\text{for every }t\geq0,\\
 \label{eq:27}
  &\frac 12\b_\e(\rho_\e(t))+t\,\a_\e(\rho_\e(t))+t^2\b_\e(\partial_t
  \rho_\e(t))\le \frac12 \b_\e(\rho_\e^0)
  &&\quad\text{for every }t>0.  
\end{alignat}
Although versions of these expressions appear in various places, we were unable to find a reference that completely suits our purposes. We therefore briefly describe their proof, and we use the more conventional formulation in terms of the bilinear forms $a_\e$ and $b_\e$ and spaces $H_\e$ and $V_\e$; note that $b_\e$ is an inner product for $H_\e$, and $b_\e + a_\e$ is an inner product for $V_\e$.

When $u_0$ is sufficiently smooth, standard results (e.g.~\cite[Chapter VII]{Brezis83}) provide the existence of a solution $u_\e\in C([0,\infty);V_\e)\cap C^\infty((0,\infty); V_\e)$, such that the functions $t\mapsto a_\e(u_\e(t))$ and $t\mapsto b_\e(\partial_tu_\e(t))$ are non-increasing; in addition, the solution operator (semigroup) $S_t$ is a contraction in $H_\e$.   For this case all three expressions can be proved by differentiation. 

In order to extend them to all $u_\e^0\in H_\e$, we note that for fixed $t>0$ the two norms on $H_\e$ given by (the square roots of)
\begin{equation}
\label{norms:He}
u_\e^0 \mapsto \frac12 b_\e(u_\e^0)
\qquad\text{and}\qquad
u_\e^0 \mapsto\frac 12b_\e(S_tu_\e^0)+\int_0^ta_\e(S_ru^0_\e)\,\d r
\end{equation}
are identical by~\eqref{eq:26} on a $H_\e$-dense subset. If we approximate a general $u_\e^0\in H_\e$ by smooth $u_{\e,n}^0$, then the sequence $u_{\e,n}^0$ is a Cauchy sequence with respect to both norms; by copying the proof of completeness of the space $L^2(0,\infty;V_\e)$ (see e.g.~\cite[Th. IV.8]{Brezis83}) it follows that the integral in~\eqref{norms:He} converges. This allows us to pass to the limit in~\eqref{eq:26}. The argument is similar for~\eqref{eq:27}, when one writes the sum of~\eqref{eq:26} and~\eqref{id:a-intb} as 
\begin{equation}
  \label{eq:11}
  \frac12 b_\e(u_\e(t)) + ta_\e(u_\e(t)) + 2\int_0^t rb_\e(\partial_t u_\e(r))\, \d r = \frac12 b_\e(u_\e^0).
\end{equation}
Finally, \eqref{eq:27} follows by \eqref{eq:11} since $r\mapsto
\b_\e(\partial_t u_\e(r))$ is non-increasing.

\subsection*{The reaction-diffusion limit}
We now adopt the same point of view to
formulate the limit reaction-diffusion system in the setting of measures. Recall that for $u\in H := L^2(\cDom,\gamma)$ we set $u^{\pm}(x):=
u(x,\pm 1)$, and thus we defined the function space
\[
V:=\big\{u\in H: u^\pm\in
W^{1,2}(\Omega)\big\},
\]
and the bilinear forms
\begin{align}
&b(u,v)=
  \frac 12\int_\Omega \Big(u^+v^++u^-v^-\Big)\,\d \lambdaO,\\
  &a(u,v):=
  \frac12\int_\Omega \Big(\nnabla_x u^+\nnabla_x v^++\nnabla_x
  u^-\nnabla_x v^-+k\big(u^+-u^-\big)(v^+-v^-)\Big)\,\d \lambdaO.
\end{align}  

As before we now extend these definitions to arbitrary measures by
\begin{align}
  \label{eq:23bis}
  \b(\rho)&:=b(u,u)&&\text{if $\rho\ll\gamma$ and $u=\frac{\d\rho}{\d\gamma}\in H$},\\
  \label{eq:24bis}
  \a(\rho)&:=a(u,u)&&\text{if $\rho\ll\gamma$ and
    $u=\frac{\d\rho}{\d\gamma}\in V$},
\end{align}
with corresponding bilinear forms $\b(\cdot,\cdot)$ and $\a(\cdot,\cdot)$; problem (\ref{eq:32},b,c) can be reformulated as 
\[
\b(\partial_t \rho(t),\sigma) + \a(\rho(t),\sigma)=0\quad
  \text{for every $t>0$ and $\sigma\in \mathsf{Dom}(a)$},
\]
or in the integral form 
\begin{equation}
  \label{eq:25_0}
  \b(\rho(t),\sigma)+\int_0^t
  \a(\rho(r),\sigma)\,\d r=\b(\rho^0,\sigma)\quad
  \text{for every }\sigma\in \mathsf{Dom}(\a).
\end{equation}

Since both problems \eqref{eq:25} and \eqref{eq:25_0} are embedded
in the same measure space $\M(\cDom)$, we can
study the convergence of the solution $\rho_\e$ of \eqref{eq:25} as
$\e\downarrow 0$.


\section{\texorpdfstring{$\Gamma$}{Gamma}-convergence result for the quadratic forms \texorpdfstring{$\a_\e,\b_\e$}{a, b}}
\label{sec:Gamma}
The aim of this section is to prove the following
$\Gamma$-convergence result:
\begin{theorem}
  \label{thm:Gamma}
  If $\rho_\e\weakto\rho$ as $\e\downarrow0$ in $\M(\cDom)$ then
  \begin{equation}
    \label{eq:35}
    \liminf_{\e\downarrow0}\a_\e(\rho_\e)\ge \a(\rho),\quad
    \liminf_{\e\downarrow0}\b_\e(\rho_\e)\ge \b(\rho).
  \end{equation}
  For every $\rho\in \M(\cDom)$ such that
  $\a(\rho)+\b(\rho)<+\infty$ there exists a family $\rho_\e\in
  \M(\cDom)$ weakly-$*$ converging to $\rho$ such that
  \begin{equation}
    \label{eq:36}
    \lim_{\e\downarrow0}\a_\e(\rho_\e)=\a(\rho),\quad
    \lim_{\e\downarrow0}\b_\e(\rho_\e)=\b(\rho).
  \end{equation}
\end{theorem}

Note that $\M(\cDom)$ endowed with the weak-$*$ topology is the dual of a
  separable Banach space, and therefore the sequential definition of
  $\Gamma$-convergence coincides with the topological
  definition~\cite[Proposition 8.1 and Theorem 8.10]{DalMaso93}; consequently Theorem~\ref{thm:Gamma} implies the $\Gamma$-convergence of the families $\a_\e$ and $\b_\e$. 
Theorem~\ref{thm:Gamma} actually states a stronger result, since the recovery sequence can be chosen to be the same for $\a_\e$ and $\b_\e$. This joint $\Gamma$-convergence of the families $\a_\e$ and $\b_\e$ is nearly equivalent with $\Gamma$-convergence of combined quadratic forms:

\begin{lemma}
\label{lemma:equivalence}
Theorem~\ref{thm:Gamma} implies the
  \begin{equation}
    {\text{ $\Gamma(\M(\cDom))$-convergence of}\quad\mathfrak
      q^\kappa_\e(\rho):=\b_\e(\rho)+\kappa\,\a_\e(\rho)\quad\text{to}\quad
    \q^\kappa(\rho):=\b(\rho)+\kappa\,\a(\rho)}
      \label{eq:6}
    \end{equation}
for each $\kappa>0$.

Conversely, if we assume~\eqref{eq:6}, then~\eqref{eq:36} holds, and~\eqref{eq:35} follows under the additional assumption
      \begin{equation}
        \label{eq:8}
        \limsup_{\e\downarrow0}\a_\e(\rho_\e)+\b_\e(\rho_\e)=C<+\infty.
      \end{equation}
\end{lemma}

\begin{proof}
The first part of the Lemma is immediate. For the second part, suppose that $\rho_\e\weakto\rho$ and satisfies \eqref{eq:8};
     the $\Gamma$-liminf inequality for $\q_\e^\kappa$
      yields
      \begin{align*}
        \liminf_{\e\downarrow0}\b_\e(\rho_\e)\ge
        \liminf_{\e\downarrow0}\q_\e^k(\rho_\e)-C\kappa\ge
        \q^\kappa(\rho)-C\kappa=
        \b(\rho)+\kappa\,\big(\a(\rho)-C\big)\quad\text{for every }\kappa>0,
      \end{align*}
      and therefore the second inequality of \eqref{eq:35} follows by
      letting $\kappa\downarrow0$. A similar argument yields the first
      inequality of \eqref{eq:35}.

      Concerning \eqref{eq:36}, $\Gamma$-convergence of $\mathfrak
      q_\e^1$ to $\q^1$ yields a recovery family
      $\rho_\e\weakto \rho$ such that
      \begin{displaymath}
        \lim_{\e\downarrow0}\a_\e(\rho_\e)+\b_\e(\rho_\e)=\a(\rho)+\b(\rho)<+\infty;
      \end{displaymath}
      In particular $\a_\e(\rho_\e)+\b_\e(\rho_\e)$ is uniformly bounded, so
      that \eqref{eq:35} yields the separate convergence \eqref{eq:36}.
\end{proof}

One of the most useful consequences of~\eqref{eq:6}
is contained in the next result (see e.g.\ \cite[Lemma
3.6]{PennacchioSavareColliFranzone05}).

\begin{corollary}[Weak-strong convergence]
  \label{cor:Gamma}
  Assume that \eqref{eq:6} holds for every $\kappa>0$ and
  let $\rho_\e,\sigma_\e\in \M(\cDom)$ be two families weakly converging to
  $\rho,\sigma$ as $\e\downarrow0$ and satisfying the uniform
  bound \eqref{eq:8}, i.e.
 \begin{equation}
    \label{eq:7}
    \limsup_{\e\downarrow0}\a_\e(\rho_\e)+\b_\e(\rho_\e)<+\infty,\quad
    \limsup_{\e\downarrow0}\a_\e(\sigma_\e)+\b_\e(\sigma_\e)<+\infty,
  \end{equation}   
  so that $\rho,\sigma$ belong to the domains of the bilinear form $\a$
  and $\b$.
    We have
  \begin{align}
    \label{eq:44}
    \lim_{\e\downarrow0}\a_\e(\sigma_\e)=\a(\sigma)\qquad\Longrightarrow&\qquad
    \lim_{\e\downarrow0}\a_\e(\rho_\e,\sigma_\e)=\a(\rho,\sigma)\\
    \lim_{\e\downarrow0}\b_\e(\sigma_\e)=\b(\sigma)\qquad\Longrightarrow&\qquad
   \lim_{\e\downarrow0}\b_\e(\rho_\e,\sigma_\e)=\b(\rho,\sigma).
  \end{align}
\end{corollary}
\begin{proof}
  We reproduce here the proof of
  \cite{PennacchioSavareColliFranzone05} in the case of the quadratic forms
  $\a_\e$ \eqref{eq:44}.
  Note that by~\eqref{eq:7} and Lemma~\ref{lemma:equivalence} we can assume that $\rho_\e$ and $\sigma_\e$ satisfy~\eqref{eq:35}.
  For every positive scalar $ r>0$ we have
  \begin{displaymath}
    2\a_\e(\rho_\e,\sigma_\e)=2\a_\e(r\,\rho_\e,r^{-1}\sigma_\e)=
    \a_\e(r\rho_\e+ r^{-1}\sigma_\e)- r^2\a_\e(\rho_\e)- r^{-2}\a_\e(\sigma_\e).
  \end{displaymath}
  Taking the inferior limit as $\e\downarrow0$ and recalling
  \eqref{eq:35} we get for $A:=\limsup_{\e\downarrow0}\a_\e(\rho_\e)$
  \begin{displaymath}
   \liminf_{\e\downarrow0} 2\a_\e(\rho_\e,\sigma_\e)\ge
   \a(r\rho+r^{-1}\sigma)-r^2
   A- r^{-2}\a(\sigma)=
   2\a(\rho,\sigma)+r^2\big(\a(\rho)-A\big).
  \end{displaymath}
  Since $r>0$ is arbitrary and $A$ is finite by
    \eqref{eq:7} we obtain $ \liminf_{\e\downarrow0}
  \a_\e(\rho_\e,\sigma_\e)\ge \a(\rho,\sigma)$ and inverting the sign
  of $\sigma$ we get \eqref{eq:44}.
\end{proof}

We split the proof of Theorem \ref{thm:Gamma} in various steps.
\subsection{Estimates near $\Omega\times\{-1,1\}$.}
\begin{lemma}
  If $\rho_\e=u_\e\gamma_\e$ satisfies the uniform bound
  $\a_\e(\rho_\e)\le C<+\infty$ for every $\e>0$, then
  for every $\delta\in (0,1)$
  \begin{equation}\label{boundpartialxi}
    \partial_\xi u_\e \to 0 \quad \mbox{in } L^2(\T \times \omega_\delta),\quad \mbox{as } \e \to 0,
  \end{equation}
  where $\omega_\delta:= (-1,-\delta)\cup (\delta,1)$.
\end{lemma}
\begin{proof}
  We observe that 
\[
\tau_\e \int_\sDom (\partial_\xi u_\e)^2\, \d\gamma_\e\le
\a_\e(\rho_\e)\le C < \infty.
\]
If $\ds h_\delta=\sup_{\xi\in \omega_\delta}H(\xi)<1$,
then $\ds \inf_{\xi\in \omega_\delta} e^{-H(\xi)/\e}=e^{-h_\delta/ \e},$ and we find
\[
\int_{\Omega \times \omega_\delta} (\partial_\xi u_\e)^2\, \d x\,\d\xi\leq
C\frac{Z_\e}{\tau_\e}e^{\frac{h_\delta}{\e}}=
C\frac{Z_\e}{\e}e^{\frac{h_\delta-1}\e}.
\]
Taking the limit as $\e\to 0$ we obtain \pref{boundpartialxi}.
\end{proof}
\begin{lemma}[Convergence of traces]
\label{le:1}
Let us suppose that $\rho_\e=u_\e\gamma_\e\weakto \rho=u\gamma$
with $\a_\e(\rho_\e)\le C<+\infty$ and let
$u_{\e}^\pm(x)$ be the traces of $u_\e$ at $\xi=\pm 1$. Then as $\e\downarrow0$
\begin{equation}\label{eq:b:11}
  u_{\e}^\pm\to u^\pm\quad \text{strongly in $L^2(\Omega)$,}
\end{equation}
where $u^\pm$ are the functions given by \eqref{eq:51}.
\end{lemma}
\begin{proof}
  Let us consider, e.g., the case of $u_\e^-$. Let us fix $\delta\in (0,1)$; by \pref{boundpartialxi} and standard trace results in $W^{1,2}(-1,-1+\delta)$ we know that
\begin{equation}\label{eq:b:14}
  \lim_{\e\downarrow0} \int_\T
  \omega^2_\e(x)\,d\Ld=0\quad
    \text{where}\quad
    \omega^2_\e(x):=\!\!\!\!\sup_{-1\le\xi\le -1+\delta}\!|u_\e(x,\xi)-
    u_\e^-(x)|^2\le \delta \!\int_{-1}^{-1+\delta} \hspace{-0.3cm} |\partial_\xi u_\e(x,\xi)|^2\d\xi.
\end{equation}
Let us fix a function $\phi\in C^0(\T)$ and a function $\psi\in C^0[-1,1]$ with $0\le \psi \le 1$, $\psi(-1)=1$, ${\rm supp}\,\psi\subset [-1,-1+\delta]$; we set 
$$J_\e:=\int_{-1}^1 \psi(\xi)\,\d\tilde\gamma_\e(\xi),\quad
\tilde u_\e(x):=J_\e^{-1}\int_{-1}^1 u_\e(x,\xi)\psi(\xi)\,d\tilde\gamma_\e(\xi),$$
where $\tilde \gamma_\e$ is the measure defined in \pref{eq:28}.
Note that
$$ \lim_{\e \to 0}J_\e = \left<\psi,\gamma \right>= \frac 12 \psi(-1)+ \frac 12 \psi(1)= \frac 12.$$
Since $\rho_\e$ weakly converge to $\rho$ we know that
\[
\lim_{\e\downarrow0} \int_\Omega \phi(x)\tilde u_\e(x)\,\d\lambdaO=
\lim_{\e\downarrow0}J^{-1}_\e\int_\Omega \phi(x) \psi(\xi)
u_\e(x,\xi)\,\d\gamma_\e(x,\xi)=\int_\Omega \phi(x)u^-(x)\,\d\lambdaO
  \]
  so that $\tilde u_\e$ converges to $u^-$ in the duality with bounded
  continuous functions. On the other hand,
  \begin{displaymath}
    \int_\Omega |\nnabla_x\tilde u_\e(x)|^2\,\d\lambdaO\le
    J_\e^{-1} \int_\Omega \int_{-1}^1 |\nnabla_x
    u_\e(x,\xi)|^2\psi(\xi)\,\d\tilde\gamma(\xi)\,\d\lambdaO(x)\le
    J_\e^{-1}\a_\e(\rho_\e)\le 2C
  \end{displaymath}
  so that $\tilde u_\e\to u^-$ in $L^2(\Omega)$ by Rellich compactness
  theorem.
  
  On the other hand, thanks to \eqref{eq:b:14}, we have
  \begin{align*}
    &\lim_{\e\downarrow0}
    \int_\Omega \Big|u_\e^-(x)-\tilde u_\e(x)\Big|^2\,\d\lambdaO(x)=
    \lim_{\e\downarrow0}J_\e^{-2}\int_\Omega
    \Big|\int_{-1}^1 \psi(\xi)
    \big(u_\e(x,\xi)-u^-(x)\big)\,\d\tilde\gamma_\e(\xi)
    \Big|^2\,\d\lambdaO(x)\\
     &\qquad\le \lim_{\e\downarrow0}
     \int_\Dom \psi(\xi)\omega_\e^2(x)\,\d\gamma_\e(x,\xi)
     =0,
  \end{align*}
  which yields \eqref{eq:b:11}.
  \drop{  \begin{align*}
    &\lim_{\e\downarrow0} 
    \Big|\int_\Omega\phi(x) \psi(\xi)
    u_\e(x,\xi)\,\d\gamma_\e(x,\xi)-
    \int_\Omega \phi(x) \psi(\xi)
     u^-_\e(x)\,\d\gamma_\e(x,\xi)\Big|\\
     &\qquad\le \lim_{\e\downarrow0}
     \int_\Omega\phi(x) \psi(\xi)\Big|
    u_\e(x,\xi)-u^-_\e(x)\Big|\,\d\gamma_\e(x,\xi)
    \le \lim_{\e\downarrow0}\,
    J_\e \int_\Omega\phi(x)\omega_\e(x)\,d\Ld(x)=0.
  \end{align*}
It follows that
  \begin{align*}
    \lim_{\e\downarrow0}
    \int_\T\phi(x)u^-_\e(x)\,\d\Ld&=
    \lim_{\e\downarrow0}J_\e^{-1}
     \int_\Omega \phi(x) \psi(\xi)
     u^-_\e(x)\,\d\gamma_\e(x,\xi)=
     \int_\T\phi(x)u^-(x)\,\d x.
  \end{align*}	}
\end{proof}
\begin{remark}
  \label{rem:after}
  \upshape
  A completely analogous argument shows that if $\rho_\e$ satisfies a
  $W^{1,1}(\Dom;\gamma_\e)$-uniform bound
  \begin{equation}
    \label{eq:65bis}
    \int_\Dom
    \|\nnabla_{x,\xi}u_\e\|_{\e}\,\d\gamma_\e(x,\xi)\le C<+\infty
  \end{equation}
  instead of $\a_\e(\rho_\e)\le C$, then $u_\e^\pm\to u^\pm$ in $L^1(\Omega)$.
\end{remark}
\subsection{Asymptotics for the minimal transition cost.}
\label{subsec:minimal_cost}
Given $(\varphi^-,\varphi^+)\in \R^2$ let us set
\begin{equation}
  \label{eq:b:6}
  K_\e(\varphi^-,\varphi^+):= \min\Big\{ \tau_\e\int_{-1}^1
  \big(\varphi'(\xi)\big)^2\,d\tilde\gamma_\e: \varphi\in
  W^{1,2}(-1,1),\ \varphi(\pm 1)=\varphi^\pm\Big\}
\end{equation}
It is immediate to check that $K_\e$ is a quadratic form
depending only on $\varphi^+-\varphi^-$, i.e.
\begin{equation}
  \label{eq:38}
  K_\e(\varphi^-,\varphi^+)=k_\e(\varphi^+-\varphi^-)^2,\quad
  k_\e=K_\e(-1/2,1/2).   
\end{equation}
We call $\mathcal T_\e(\varphi^-,\varphi^+)$ the solution of the minimum
problem \eqref{eq:b:6}: it admits the simple representation
\begin{equation}
  \label{eq:39}
  \mathcal T_\e(\varphi^-,\varphi^+)=\frac12(\varphi^-+\varphi^+)+(\varphi^+-\varphi^-)\phi_\e
\end{equation}
where $\phi_\e=\mathcal T_\e(-1/2,1/2)$.
We also set
\begin{equation}
  \label{eq:42}
   Q_\e(\varphi^-,\varphi^+):=\int_{-1}^1 \big(\mathcal
    T_\e(\varphi^-,\varphi^+)\big)^2\,\d\tilde\gamma_\e=
    \frac12\big((\varphi^-)^2+(\varphi^+)^2\big)+(q_\e-\tfrac14)(\varphi^+-\varphi^-)^2
\end{equation}
where
\begin{equation}
  \label{eq:43}
  q_\e:=\int_{-1}^1 |\phi_\e(\xi)|^2\,\d\tilde\gamma_\e(\xi)=Q_\e(-1/2,1/2).
\end{equation}
\begin{lemma}
  We have
  \begin{equation}
    \label{eq:40}
    \lim_{\e\downarrow0}k_\e=\frac k2=\frac {\sqrt{-H''(0)\,H''(1)} }{2\pi},
  \end{equation}
  and
  \begin{equation}
    \label{eq:41}
    \lim_{\e\downarrow0}q_\e=\frac 14\quad\text{so that}\quad
    \lim_{\e\downarrow0} Q_\e(\varphi^-,\varphi^+)=
    \frac 12 (\varphi^-)^2+\frac 12(\varphi^+)^2.
  \end{equation}
\end{lemma}
\begin{proof}
  $\phi_\e$ solves the Euler equation
  \begin{equation}
    \label{eq:euler}
    \big(e^{-H(\xi)/\e}\phi_\e'(\xi)\big)'=0\quad\text{on
    }(-1,1),\quad
    \phi_\e(\pm 1)=\pm \tfrac 12.
  \end{equation}
  We can compute an explicit solution of \pref{eq:euler} by
  integration:
  \begin{align*}
    \phi_\e'(\xi) = Ce^{H(\xi)/\e},\qquad \phi_\e(\xi) = C' +
    C\int_0^\xi e^{H(\eta)/\e}\, \d\eta.
  \end{align*}
  Define $I_\e:=\int_{-1}^1e^{H(\xi)/\e}\, \d\xi$. The boundary
  conditions for $\xi=\pm1$ give
  \[
  C' =0, \quad C\int_{-1}^1e^{H(\xi)/\e}\, d\xi = C I_\e=1
  \]
  It follows that
  \[
  \phi_\e(\xi) = I_\e^{-1} \int_0^\xi e^{H(\eta)/\e}\, \d\eta,
  \]
  and
  \begin{align*}
    k_\e=\tau_\e I_\e^{-2}\int_{-1}^1
    e^{2H(\xi)/\e}\,\d\tilde\gamma_\e(\xi)= \tau_\e Z^{-1}I_\e^{-1}.
  \end{align*}
  We compute, using Laplace's method:
  $$ I_\e  =  \sqrt{\frac{2\pi\e}{|H''(0)|}}e^{1/\e}(1+o(1))\quad \mbox{and}\quad Z_\e  = \sqrt{\frac{2\pi\e}{H''(1)}}(1+o(1)),\quad \mbox{as } \e \to 0,$$ 
  thus obtaining \eqref{eq:40}.
   Since
  \[
  \phi'_\e=I_\e^{-1}e^{H/\e}\to \delta_0\quad \mbox{ in }\mathscr{D}'(-1,1)
  \]
  and $H$ is even, we have
  \[
  \phi_\e(\xi)=
  I_\e^{-1} \int_0^\xi e^{H(\eta)/\e}\, \d\eta\to \frac12{\rm
    sign}(\xi)
  \]
  uniformly on each compact subset of $[-1,1]$ not containing
  $0$. Since the range of $\phi_\e$ belongs to $[-1/2,1/2]$ and
  $\tilde\gamma_\e\weakto \frac12\delta_{-1}+\frac 12\delta_{+1}$ we
  obtain \eqref{eq:41}.
\end{proof}
\subsection{End of the proof of Theorem \ref{thm:Gamma}.}
The second limit of \eqref{eq:35} follows by general
lower semicontinuity results on integral functionals of measures,
see e.g.\ \cite[Lemma 9.4.3]{AmbrosioGigliSavare05}.

Concerning the first ``$\liminf$'' inequality,
we split the the quadratic form $\a_\e$ in the sum of two parts,
\begin{equation}
  \label{eq:42bis}
  \a_\e^1(\rho_\e):=\int_\sDom |\nnabla_x
  u_\e(x,\xi)|^2\,\d\gamma_\e(x,\xi),\quad
  \a_\e^2(\rho_\e):=\tau_\e\int_\sDom(\partial_\xi u_\e)^2\,\d\gamma_\e(x,\xi).
\end{equation}
We choose a smooth cutoff function $\eta^-:[-1,1]\to [0,1]$ such that
$\eta^-(-1)=1$ and $\mathrm{supp}(\eta^-)\subset [-1,-1/2]$ and the
symmetric
one $\eta^+(\xi):=\eta(-\xi)$.
We also set
\begin{equation}
  \label{eq:37}
  u_\e^-(x):=\int_{-1}^1\eta^-(\xi)u_\e(x,\xi)\,\d\tilde\gamma_\e(\xi),\quad
  u_\e^+(x):=\int_{-1}^1\eta^+(\xi)u_\e(x,\xi)\,\d\tilde\gamma_\e(\xi),\quad
\end{equation}
and it is easy to check that
\begin{equation}
  \label{eq:43bis}
  u_{\e}^\pm\weakto \frac 12 u^\pm\quad\text{in }\mathscr D'(\Omega).
\end{equation}
We also set $\theta_\e:=\int_{-1}^1
\eta^+(\xi)\,\d\tilde\gamma_\e(\xi)$ $\big(= \int_{-1}^1\eta^-(\xi)\,\d\tilde\gamma_\e(\xi)\big)$, observing that $\theta_\e\to
1/2$.
We then have by Jensen inequality
\begin{align*}
  \a_\e^1(\rho_\e)&\ge \int_\Omega \int_{-1}^1
  (\eta^-(\xi)+\eta^+(\xi))|\nabla_x
  u_\e|^2\,\d\tilde\gamma_\e(\xi) \,\d \lambdaO\ge \theta_\e^{-1}
  \int_\Omega |\nabla u_\e^-|^2+|\nabla u_\e^+|^2\,\d \lambdaO
\end{align*}
and, passing to the limit,
\begin{displaymath}
  \liminf_{\e\downarrow0}\a_\e^1(\rho_\e)\ge
  \frac 12\int_\Omega |\nabla u^{-}|^2+|\nabla u^{+}|^2\,\d \lambdaO
\end{displaymath}
Let us now consider the behaviour of $\a^2_\e$: applying
\eqref{eq:b:6} and \eqref{eq:38} we get
\begin{align*}
  \a^2_\e(\rho_\e)&=\int_\Omega \left(\tau_\e \int_{-1}^1 (\partial_\xi
  u_\e(x,\xi))^2\,\d\tilde\gamma_\e(\xi)\right)\,\d \lambdaO\ge \int_\Omega
k_\e(u_\e^-(x)-u_\e^+(x))^2\,\d \lambdaO
\end{align*}
so that by \eqref{eq:40} and \eqref{eq:b:11} we obtain
\begin{equation}
  \label{eq:44bis}
  \liminf_{\e\downarrow0}\a^2_\e(\rho_\e)\ge \frac k2\int_\Omega
  \big(u^-(x)-u^+(x)\big)^2\,\d\lambdaO.
\end{equation}
In order to prove the ``$\limsup$'' inequality \eqref{eq:36} we fix
$\rho=u\gamma$ with $u$ in the domain of  the quadratic forms $a$ and
$b$ so that $u^\pm=u(\cdot,\pm1)$ belong to $W^{1,2}(\Omega)$, and we
set
$\rho_\e=u_\e\gamma_\e$ where $u_\e(x,\cdot)=\mathcal
T_\e(u^-(x),u^+(x))$ as in \eqref{eq:39}.
We easily have by \eqref{eq:41} and the Lebesgue dominated convergence
theorem 
\begin{align*}
  \lim_{\e\downarrow0}\b_\e(\rho_\e)&=
  \lim_{\e\downarrow0}\int_\Omega Q_\e(u^-(x),u^+(x))\,\d\lambdaO=
  \int_\Omega \Big(\frac12 |u^-(x)|^2+\frac 12|u^+(x)|^2\Big)\,\d\lambdaO=
  \b(\rho).
\end{align*}
Similarly, since for every $j=1,\cdots, d$ and almost every $x\in
\Omega$
$$\partial_{x_j} u_\e(x,\xi)=\mathcal
T(\partial_{x_j}u^-(x),\partial_{x_j} u^+),$$ 
we have
\begin{align*}
  \lim_{\e\downarrow0}\a_\e(\rho_\e)&=
  \lim_{\e\downarrow0}\int_\Omega
  \Big(\sum_{j=1}^d  Q_\e\big(\partial_{x_j}u^-(x),\partial_{x_j} u^+(x)\big)+K_\e\big(u^-(x),u^+(x)\big)\Big)\,\d\lambdaO=\\&=
  \int_\Omega \Big(\frac12 |\nabla u^-(x)|^2+\frac 12|\nabla
  u^+(x)|^2+\frac k2\big(u^-(x)-u^+(x)\big)\Big)\,\d\lambdaO=
  \a(\rho).
\end{align*}


\section{From \texorpdfstring{$\Gamma$}{Gamma}-convergence to convergence of the evolution
  problems: proof of Theorems 1 and 2.}

\label{sec:proof}
Having at our disposal the $\Gamma$-convergence result of Theorem
\ref{thm:Gamma} and its Corollary \ref{cor:Gamma}
it is not difficult to pass to the limit in the integrated equation
\eqref{eq:25}.

Let us first notice that the quadratic forms $\b_\e$ satisfy a
uniform coercivity condition:
\begin{lemma}[Uniform coercivity of $\b_\e$]
  \label{le:coercivity}
  Every family of measures $\rho_\e\in \M(\cDom)$, $\e>0$
  satisfying
    \begin{equation}
      \label{eq:9}
      \limsup_{\e>0}\b_\e(\rho_\e)<+\infty
    \end{equation}
    is bounded in $\M(\cDom)$ and admits a weakly-$*$ converging subsequence.
  \end{lemma}
  \begin{proof}
    It follows immediately by the fact that $\gamma_\e$ is a
    probability measure and therefore
    \begin{displaymath}
      |\rho_\e|(\cDom)\le \Big(\b_\e(\rho_\e)\Big)^{1/2}.
    \end{displaymath}
    \eqref{eq:9} thus implies that the total mass of $\rho_\e$ is
    uniformly bounded and we can apply the relative weak-$*$ compactness of
    bounded sets in 
    dual Banach spaces.
  \end{proof}
  The proof of Theorems \ref{th:1} and \ref{th:2} is a
  consequence of the following general result:
  \begin{theorem}[Convergence of evolution problems]
    \label{thm:evolconv}
    Let us consider weakly-$*$ lower-semi\-conti\-nuous, nonnegative
    and extended-valued
    quadratic forms $\a_\e,\b_\e, \a,\b$ defined on
    $\M(\cDom)$ and let us suppose that
    \begin{enumerate}[\bf 1)]
    \item \textbf{Non degeneracy of the limit forms:}
      $\b$ is non degenerate (i.e.~$\b(\rho)=0\
      \Rightarrow\ \rho=0$) and $\Domain(\a)$ is dense in
      $\Domain(\b)$ with respect to the norm-convergence induced by $\b$.      
    \item \textbf{Uniform coercivity:} $\b_\e$ satisfy the coercivity property stated in the previous
      Lemma \ref{le:coercivity}.
    \item \textbf{Joint $\Gamma$-convergence:}
      $\q_\e^\kappa:=\b_\e+\kappa\,\a_\e$ satisfy the joint $\Gamma$-convergence property
      \eqref{eq:6}
      \begin{equation}
        \label{eq:10}
        \Gamma\big(\M(\cDom)\big)\text{-}\lim_{\e\downarrow0}\q_\e^\kappa=
        \q^\kappa=\b+\kappa\,\a\quad\text{for every }\kappa>0.
      \end{equation}
    \end{enumerate}
    Let $\rho_\e(t)$, $t\ge0$, be the solution of the evolution problem
    \eqref{eq:25} starting from $\rho_\e^0\in \mathsf{Dom}(\b_\e)$.

    If
    \begin{equation}
      \label{eq:2}
      \text{$\rho^0_\e\weakto\rho^0$
        in $\M(\cDom)$ as $\e\downarrow0$\quad
        with}\quad
      \limsup_{\e\downarrow0}\b_\e(\rho_\e^0)<+\infty  
    \end{equation}
    then $\rho_\e(t)\weakto\rho(t)$ in $\M(\cDom)$ as $\e\downarrow0$
    for every $t>0$ and $\rho(t)$ is the solution of the limit
    evolution problem \eqref{eq:25_0}.

    If moreover $\lim_{\e\downarrow0}\b_\e(\rho_\e^0)=\b(\rho_0)$ then
    \begin{equation}
      \label{eq:3}
      \lim_{\e\downarrow0}\b_\e(\rho_\e(t))=\b(\rho(t)),\quad
      \lim_{\e\downarrow0}\a_\e(\rho_\e(t))=\a(\rho(t))\quad
      \text{for every }t>0.
    \end{equation}
  \end{theorem}
  
\begin{proof}
Let us first note that by \eqref{eq:26} and the coercivity
  property of $\b_\e$ 
the mass of $\rho_\e(t)$
is bounded uniformly in $t$. Moreover, \eqref{eq:27} and the
  coercivity property show that
$\partial_t\rho_\e$ is a finite measure whose total
mass is uniformly bounded in each bounded interval
$[t_0,t_1]\subset (0,+\infty)$. By the Arzela-Ascoli theorem we can extract
a subsequence $\rho_{\e_n}$ such that $\rho_{\e_n}(t)\weakto \rho(t)$
for every $t\ge0$. The estimates~\eqref{eq:27} and \eqref{eq:35} show that
for every $t>0$, $\rho(t)$ belongs to the domain of the quadratic forms
$\a$ and $\b$, and satisfies a similar estimate
\begin{equation}
  \label{eq:45bis}
  \frac12\b(\rho(t))+t\,\a(\rho(t))+t^2\b(\partial_t \rho(t))\le
  \frac12\liminf_{\e\downarrow0}\b(\rho^0_\e)<+\infty.
\end{equation}

Let $\sigma\in \M(\cDom)$ be an arbitrary element of the domains of
$\a$ and $\b$; by \eqref{eq:6} we can find
a family $\sigma_\e$ (actually a family $\sigma_{\e_n}$, but we suppress the subscript $n$) weakly converging to $\sigma$ such that
\eqref{eq:36} holds.
By \eqref{eq:25} we have
\begin{equation}
  \label{eq:46}
  \b_\e(\rho_\e(t),\sigma_\e)+\int_0^t
  \a_\e(\rho_\e(r),\sigma_\e)\,\d r=\b_\e(\rho_\e^0,\sigma_\e)
\end{equation}
and \eqref{eq:27} with the Schwarz inequality yields the uniform bound
\begin{displaymath}
  \big|\a_\e(\rho_\e(t),\sigma_\e)\big|\le
  t^{-1/2}\b_\e(\rho_\e^0)^{1/2}\a_\e(\sigma_\e)^{1/2}\le Ct^{-1/2}
\end{displaymath}
where $C$ is independent of $\e$;
we can therefore pass to the limit in \eqref{eq:46} by Corollary \ref{cor:Gamma} to find
\begin{displaymath}
  \b(\rho(t),\sigma)+\int_0^t
  \a_\e(\rho(r),\sigma)\,\d r=\b(\rho_0^0,\sigma),
\end{displaymath}
so that $\rho$ is a solution of the limit equation. Since the limit is
uniquely identified by the non-degeneracy and density condition 1), we conclude that the whole family $\rho_\e$
converges to $\rho$ as $\e\downarrow0$.
In particular $\rho$ satisfies the identity
\begin{equation}
\label{id:limit-dissipation}
  \frac 12\b(\rho(t))+\int_0^t\a(\rho(r))\,\d r=\frac
  12\b(\rho^0)\quad\text{for every }t\ge0.
\end{equation}
This concludes the proof of \eqref{eq:2} (and of Theorem \ref{th:1}).

In order to prove \eqref{eq:3} (and Theorem \ref{th:2}) we note that by
  \eqref{eq:26} and \eqref{id:limit-dissipation} we easily get
\begin{displaymath}
  \limsup_{\e\downarrow0}\frac 12\b_\e(\rho_\e(t))+\int_0^t\a_\e(\rho_\e(r))\,\d r
  \le
  \frac 12\b(\rho(t))+\int_0^t\a(\rho(r))\,\d r.
\end{displaymath}
The lower-semicontinuity property
\eqref{eq:35} and Fatou's Lemma yield 
\begin{equation}
  \label{eq:47}
  \lim_{\e\downarrow0}\b_\e(\rho_\e(t))= \b(\rho(t)),\quad
  \lim_{\e\downarrow0}\int_0^t\a_\e(\rho_\e(r))\,\d r
  =
  \int_0^t\a(\rho(r))\,\d r\quad\text{for
    every }t\ge0.
\end{equation}
Applying the same argument to~\eqref{id:a-intb} and its ``$\e=0$'' analogue 
we conclude that $\a_\e(\rho_\e(t))\to \a(\rho(t))$ for every $t>0$.
\end{proof}
  \begin{remark}[More general ambient spaces]
    The particular structure of $\M(\cDom)$ did not play any role in
    the previous argument, so that the validity of the above result can be easily extended to
    general topological vector spaces (e.g.~dual of separable Banach
    spaces with their weak-$*$ topology),
    once the coercivity condition
    of Lemma \ref{le:coercivity} is satisfied.
  \end{remark}
\small
\bibliography{ref}

\begin{thebibliography}{10}

\bibitem{AdamsDirrPeletierZimmerInPrep}
S.~Adams, N.~Dirr, M.~A. Peletier, and J.~Zimmer.
\newblock Foundation of the {W}asserstein gradient-flow formulation of
  diffusion: A large-deviation approach.
\newblock In preparation.

\bibitem{AmbrosioGigliSavare05}
L.~Ambrosio, N.~Gigli, and G.~Savar\'e.
\newblock {\em Gradient Flows in Metric Spaces and in the Space of Probability
  Measures}.
\newblock Lectures in mathematics ETH Z\"urich. Birkh\"auser, 2005.

\bibitem{Aris99}
R.~Aris.
\newblock {\em {Mathematical Modeling: A Chemical Engineer's Perspective}}.
\newblock Academic Press, 1999.

\bibitem{Attouch84}
H.~Attouch.
\newblock {\em Variational convergence for functions and operators}.
\newblock Pitman (Advanced Publishing Program), Boston, MA, 1984.

\bibitem{Brezis73}
H.~Brezis.
\newblock {\em Op{\'e}rateurs maximaux monotones et semi-groupes de
  contractions dans les espaces de {H}ilbert}.
\newblock North Holland, 1973.

\bibitem{Brezis83}
H.~Brezis.
\newblock {\em Analyse fonctionelle - Th\'eorie et applications}.
\newblock Collection Math\'ematiques appliqu\'ees pour la ma\^itrise. Masson,
  first edition, 1983.

\bibitem{DaPRatoLunardi07}
G.~Da~Prato and A.~Lunardi.
\newblock On a class of self-adjoint elliptic operators in {$L^2$} spaces with
  respect to invariant measures.
\newblock {\em J. Differential Equations}, 234(1):54--79, 2007.

\bibitem{DalMaso93}
G.~Dal~Maso.
\newblock {\em An introduction to {$\Gamma$}-convergence}, volume~8 of {\em
  Progress in Nonlinear Differential Equations and Their Applications}.
\newblock Birkh\"auser, Boston, first edition, 1993.

\bibitem{ErdiToth89}
P.~{\'E}rdi and J.~T{\'o}th.
\newblock {\em {Mathematical Models of Chemical Reactions: Theory and
  Applications of Deterministic and Stochastic Models}}.
\newblock Manchester University Press, 1989.

\bibitem{Hutchinson86}
J.~E. Hutchinson.
\newblock Second fundamental form for varifolds and the existence of surfaces
  minimising curvature.
\newblock {\em Indiana Univ. Math. J.}, 35:45--71, 1986.

\bibitem{JordanKinderlehrerOtto98}
R.~Jordan, D.~Kinderlehrer, and F.~Otto.
\newblock The variational formulation of the {F}okker-{P}lanck {E}quation.
\newblock {\em SIAM Journal on Mathematical Analysis}, 29(1):1--17, 1998.

\bibitem{Kramers40}
H.~A. Kramers.
\newblock {Brownian motion in a field of force and the diffusion model of
  chemical reactions}.
\newblock {\em Physica}, 7(4):284--304, 1940.

\bibitem{LionsMagenes72}
J.~Lions and E.~Magenes.
\newblock {\em Non Homogeneous Boundary Value Problems and Applications},
  volume~I.
\newblock Springer, New York-Heidelberg, 1972.

\bibitem{MielkeInPrep}
A.~Mielke.
\newblock Energy-drift-diffusion equations with recombination as
  entropy-gradient system.
\newblock In preparation, 2009.

\bibitem{Otto01}
F.~Otto.
\newblock The geometry of dissipative evolution equations: The porous medium
  equation.
\newblock {\em Comm. PDE}, 26:101--174, 2001.

\bibitem{PennacchioSavareColliFranzone05}
M.~Pennacchio, G.~Savar{\'e}, and P.~Colli~Franzone.
\newblock Multiscale modeling for the bioelectric activity of the heart.
\newblock {\em SIAM J. Math. Anal.}, 37(4):1333--1370 (electronic), 2005.

\bibitem{Reshetnjak68}
J.~G. Re{\v{s}}etnjak.
\newblock The weak convergence of completely additive vector-valued set
  functions.
\newblock {\em Sibirsk. Mat. \u Z.}, 9:1386--1394, 1968.

\end{thebibliography}
\bibliographystyle{abbrv}

\end{document}